\documentclass[12pt,reqno]{amsart}
\usepackage{amscd,amsmath,amsthm,amssymb}
\usepackage{url}
\usepackage{tikz}
\usetikzlibrary{knots,calc,shapes.geometric}
\usepackage{amsfonts,amsmath,mathtools}
\usepackage{graphics}
\usepackage{subfig}
\usepackage{xcolor}
\usepackage{latexsym}
\usepackage{graphics}
\usepackage{float}
\usepackage{enumitem}

\newtheorem{Theorem}{Theorem}[section]
\newtheorem{Lemma}[Theorem]{Lemma}
\newtheorem{Corollary}[Theorem]{Corollary}
\newtheorem{Proposition}[Theorem]{Proposition}
\newtheorem{Remark}[Theorem]{Remark}

\newtheorem{Example}[Theorem]{Example}

\newtheorem{Definition}[Theorem]{Definition}
\newtheorem{Setup}[Theorem]{Setup}

\newtheorem{Open question}[Theorem]{Open question}

\def\qed{\ifhmode\textqed\fi
	\ifmmode\ifinner\hfill\quad\qedsymbol\else\dispqed\fi\fi}
\def\textqed{\unskip\nobreak\penalty50
	\hskip2em\hbox{}\nobreak\hfill\qedsymbol
	\parfillskip=0pt \finalhyphendemerits=0}
\def\dispqed{\rlap{\qquad\qedsymbol}}

\let\epsilon\varepsilon

\def\B{{\mathcal B}}
\def\K{{\mathcal K}}
\def\T{{\mathcal T}}
\def\X{{\mathcal X}}
\def\Y{{\mathcal Y}}
\def\W{{\mathcal W}}
\def\Z{{\mathcal Z}}
\def\PG{\textup{PG}}
\def\AG{\textup{AG}}

\begin{document}
\title{Generalizing blocking semiovals in finite projective planes}
\author{Marilena Crupi, Antonino Ficarra}

\address{Marilena Crupi, Department of mathematics and computer sciences, physics and earth sciences, University of Messina, Viale Ferdinando Stagno d'Alcontres 31, 98166 Messina, Italy}
\email{mcrupi@unime.it}

\address{Antonino Ficarra, BCAM -- Basque Center for Applied Mathematics, Mazarredo 14, 48009 Bilbao, Basque Country -- Spain}
\address{Ikerbasque, Basque Foundation for Science, Plaza Euskadi 5, 48009 Bilbao, Basque Country -- Spain}
\email{aficarra@bcamath.org,\,\,\,\,\,\,\,\,\,\,\,\,\,antficarra@unime.it}




\begin{abstract}
Blocking semiovals and the determination of their (minimum) sizes constitute one of the central research topics in finite projective geometry. In this article we introduce the concept of blocking set with the $r_\infty$-property in a finite projective plane $\PG(2,q)$, with $r_\infty$ a line of $\PG(2,q)$ and $q$ a prime power. This notion greatly generalizes that of blocking semioval. We address the question of determining those integers $k$ for which there exists a blocking set of size $k$ with the $r_\infty$-property. To solve this problem, we build new theory which deeply analyzes the interplay between blocking sets in finite projective and affine planes.
\end{abstract}
\maketitle

\section{Introduction}
Blocking sets are a fundamental concept in combinatorial geometry and finite geometry, and they play a critical role in various areas of mathematics, including design theory, coding theory, and the study of finite projective and affine planes. Their intriguing properties and wide range of applications make blocking sets a rich subject of study in both theoretical and applied mathematics.

Indeed, blocking sets can be seen as a bridge between theoretical combinatorial concepts and practical applications in various fields of mathematics and engineering. Their study not only enhances our understanding of finite geometries but also contributes to advancements in technology and science through their applications in coding theory, design theory, and beyond.

Roughly speaking a \emph{blocking set} in a finite projective plane is defined as a set of points such that every line in the plane intersects the set in at least one point. More formally, consider a projective plane $\PG(2, q)$ of order $q$, where $q$ is a prime power. A subset $\B$ of the point set of $\PG(2, q)$  is called a blocking set if every line in $\PG(2, q)$ contains at least one point from $\B$ and does not contain any line. 

Such a definition can also be introduced in the context of the finite affine plane $\AG(2, q)$ where $q$ is a prime power.

To visualize this, imagine a finite projective plane where lines and points follow specific incidence properties. A blocking set ensures that no matter which line one chooses, there will always be at least one point from the blocking set on that line. This simple yet powerful property underpins the utility of blocking sets in various applications.

A crucial question in the study of blocking sets concerns determining their minimum size. For a projective plane of order $q$, it is known that any blocking set contains at least $q + \sqrt{q} + 1$ points  \cite{ABruen71}. This bound is significant because it represents the smallest number of points needed to ensure the blocking property. Understanding these minimal blocking sets is essential because they provide insight into the most efficient ways to achieve the blocking condition. A minimal blocking set has the property that it does not contain any proper subset which is also a blocking set.

In a projective plane, a \emph{semioval} is a set of points $\B$ such that there is a unique tangent
line, that is, a line with one point of contact, at each point. A \emph{blocking semioval} is a set of points in a projective plane that is both a blocking set and a semioval. Hubaut \cite{XH1970} has shown that a semioval contains at least $q + 1$ points. For a blocking semioval $\B$ we have $q+ \sqrt{q} + 1 \le \vert \B\vert\le q\sqrt{q}+1$ \cite{ABruen71, XH1970} (see also \cite{JMD00}).

The notion of semiovals has been around since the 1970's (\cite{FB73, JT74, GK}), but the study of semiovals has been motivated by the pioneering paper of Batten \cite{Bt00} and initiated by Dover in \cite{JMD00, JMD002}. While the study of blocking semiovals was originally motivated by Batten \cite{Bt00} in connection with cryptography, other authors studied these objects, because they are interesting in their own right. 

A classical example of blocking semioval is given by the vertexless triangle in a finite projective plane of order $q>2$. If $a, b, c$ are any three non concurrent lines in the plane, the vertexless triangle $\T$ is defined as the set of all points which lie on exactly one of these lines, that is, the set of points on the sides of this triangle without the vertices. $\T$ is a blocking semioval of size $3q-3$.

The problem of determining blocking semiovals of a given size is open. Moreover, the question of what the true lower bound on the size of a blocking semioval remains to be answered, since the upper bound $3(q-1)$ is reached by any vertexless triangle. Furthermore, given $q$ and $k$, determining whether there exists a blocking semioval of size $k$ in $\PG(2,q)$ is a tremendous tasks. For instance see \cite{BFKMP14,Bt00,BD01,BP18,BCF24,BS78,ABruen71,JMD00,JMD002,JM12,JMDMW16,JMD23,DMW13,XH1970,Jamison,GK,GK08,PP02,PS02,CS00,CS002,JT74} and the references therein.

Let $\PG(2, q)$ be a projective plane. In the present paper, inspired by the notion of blocking semioval, we introduce the concept of \textit{blocking set with the $r_{\infty}$-property}. A blocking set $\B$ in $\PG(2, q)$ is said to have the $r_{\infty}$-property with respect to a point $P\in \B$ if there exists through $P$ only one tangent to $\B$, called $r_{\infty}$, and the other lines through $P$ are secants to $\B$ (Definition \ref{Def:rproperty}).

A blocking semioval $\B$ satisfies the $r_{\infty}$-property with respect to every point of $\B$. Hence, the family of the blocking semiovals is a subset of the family of blocking sets with the $r_{\infty}$-property. In particular, any vertexless triangle has the $r_{\infty}$-property.

The paper is organized as follows. 

In Section \ref{secPG:2}, we collect some basic results on blocking sets and blocking semiovals.

Section \ref{secPG:3} introduces the concept of blocking set $\mathcal{B}$ with the $r_\infty$-property with respect to a point $P\in\mathcal{B}$. As we remarked before any blocking semioval, and so any vertexless triangle, is a blocking set $\mathcal{B}$ with the $r_\infty$-property with respect to any point $P\in\mathcal{B}$. We first examine the existence of these objects in the basic but motivating example $\PG(2,3)$. It is observed in Corollary \ref{cor:PG(2, 3)} that any minimal blocking set in $\PG(2,3)$ is a vertexless triangle and satisfies the $r_\infty$-property.

Next, we consider the existence of blocking sets with the $r_\infty$-property in a projective plane $\PG(2,q)$ for $q\ge4$ a prime power. As in the case of blocking semiovals, the main question we want to address, is to determine all the possible sizes $k$ of a blocking set with the $r_\infty$-property in $\PG(2,q)$. By Remark \ref{Rem:vertexlesstriangle}, any vertexless triangle is a blocking set with the $r_\infty$-property having size $k=3(q-1)$. Therefore, we address the case $k\le 3q-4$. For this investigation, we distinguish the three possible cases: (a) $2q\le k\le 3q-5$, (b) $1\le k\le 2q-1$ and (c) $k=3q-4$. 

For the case (a), we consider some special constructions due to Innamorati and Maturo \cite{IM91}. Following the presentation given in \cite[Theorem 13.15]{JH98}, in Setup \ref{setup} we describe some special blocking sets $\mathcal{B}_k$ in $\PG(2,q)$ of size $2q\le k\le 3q-5$, and we say that each $\B_k$ is built with a \textit{$k$-construction}. In Lemma \ref{lem1}, we prove that $\B_k$ is a blocking set with the $r_\infty$-property which is not a semioval, for any $2q\le k\le 3q-5$. On the other hand, we prove in Lemma \ref{lem2} that a blocking set $\B$ built with a $(2q-1)$-construction does not satisfy the $r_\infty$-infinity property for any point $P\in\B$. Combining these two lemmas we show in Theorem \ref{thm: 5} that a blocking set with the $r_\infty$-property of size $k$ exists for any $2q\le k\le3(q-1)$ with $q\ge5$.

Surprisingly, in case (b), no blocking set of size $k$ satisfies the $r_\infty$-property as shown in Theorem \ref{thm:no} of Section \ref{secPG:4}. To prove this, setting $\ell = r_{\infty}$, and defining the affine plane $\AG(2, q)$ as $\PG(2, q)\setminus \ell$, we introduce in Definition \ref{def:prodproperty} the concept of blocking set $\B$ in $\AG(2, q)$ with the \textit{$\Pi$-property}. In Proposition \ref{propo:alpha}, we establish a bijection between the set $\X_{P, \ell}$ of all minimal blocking sets in $\PG(2, q)$ with the $r_{\infty}$-property with respect to $P\in r_\infty= \ell$, and the set $\Y_{\Pi_P}$ of all minimal blocking sets in $\AG(2, q)$ with the $\Pi$-property with respect to some direction $\Pi_P$. Using this result and the key inequality $|\B|\ge 2q-1$, stated in Lemma \ref{lem:[7]}, known as the Jamison, Brouwer-Schrijver Theorem \cite{Jamison, BS78}, valid for any minimal blocking set $\B$ in $\AG(2,q)$, we prove Theorem \ref{thm:no}.  

In the last section, we discuss the case (c). If $q\ge5$, it follows from Theorem \ref{thm:size}(b), see also \cite{CS002}, that blocking sets with the $r_\infty$-property having size $3q-4$ exist. Moreover, for $q=4$, we give an example which shows that there exists a blocking set with the $r_\infty$-property of size $3q - 4=8$ in $\PG(2,q)$. In general, it is an open question to determine those integers $2q+2\le k\le 3q-5$ for which blocking semiovals of size $k$ exist. For this problem, if $2q\le k\le 3q-5$, by Lemma \ref{lem1} one has to consider blocking sets with the $r_\infty$-property which are not built with a $k$-construction. For $k>3(q-1)$ we do not know whether there exist blocking sets with the $r_\infty$-property in $\PG(2,q)$, $q>5$, having size $k$. 

\section{Generalities on blocking sets}\label{secPG:2}
In this section for the reader’s convenience we collect some notions and results we need for the development of the article.

Let $\PG(2, q)$ be the classical projective plane of order $q$ for $q$ a prime power. It is well-known that $\PG(2, q)$ has $q^2+q+1$ points, $q^2+q+1$ lines, each line passes through $q+1$ points, and each pair of distinct points lies on exactly one line.

We fix some notation. In $\PG(2, q)$, a line will be indicated with $r$ or $AB$, if $A$ and $B$ are points of $r$, and $P=\{P\}$ if $\{P\}$ is a singleton. 

Given a projective plane $\PG(2, q)$ and a line $r_{\infty}$ of $\PG(2, q)$, we define the affine plane 
$\AG(2, q)= \PG(2, q)\setminus r_{\infty}$ as follows:
\begin{enumerate}
\item[-] the points of $\AG(2, q)$ are the points of $\PG(2, q)$ that are not in $r_{\infty}$,
\item[-] the lines of $\AG(2, q)$ are the lines of $\PG(2, q)$, except $r_{\infty}$,
\item[-] $P\in\ell$ in $\AG(2, q)$ if and only if $P\in\ell$ in $\PG(2, q)$.
\end{enumerate}

Let $\K$ be a subset of $\PG(2, q)$. A \emph{tangent} to $\K$ is a line which intersects $\K$ in only one point. A \emph{secant} to $\K$ is a line which intersects $\K$ in more than one point. We note that the term tangent is used only to denote one point contact; these lines may not be tangents in the algebraic geometry sense.

We quote the next definitions from \cite[Chapter 13]{JH98}. 
\begin{Definition}\label{def1}
A blocking set of $\PG(2,q)$ is a set of points $\B$ which meets every line but does not contain any line. A blocking set $\B$ in $\AG(2,q)$ is a set of points which meets every line of $\AG(2,q)$.

A  blocking set $\B$ of $\PG(2,q)$ $(\AG(2,q)$, respectively\textup{)} is called minimal if no proper subset of $\B$ is a blocking set of $\PG(2,q)$ $(\AG(2,q)$, respectively\textup{)}.
\end{Definition}

Note that in the affine case $\B$ may contain some line of $\AG(2,q)$.

Geometrically, by \cite[Lemma 13.1]{JH98} a blocking set $\B$ in $\PG(2,q)$ is minimal if and only if, for every point $P$ of $\B$, there exists at least a tangent $\ell$ to $\B$ passing through $P$, that is some line $\ell$ such that $\B\cap\ell=P$.

\begin{Definition}\label{def2}
A semioval is a set $\K$ of points of $\PG(2, q)$ such that for every $P\in \K$ there exists a unique line $\ell$ of  $\PG(2, q)$ such that $\ell \cap \K=P$.
\end{Definition}

Combining the Definitions \ref{def1} and \ref{def2}, we obtain the concept of \emph{blocking semioval}, that is, a set of points in $\PG(2, q)$ which is both a semioval and blocking set.

One can observe that blocking semiovals are necessarily minimal blocking sets, as deleting a point of a blocking semioval $\K$ will cause that the tangent to $\K$ at that point is unblocked. On the other hand, a blocking semioval is also a maximal semioval. Indeed, adding any point to a blocking semioval $\K$ will cause the added point to have no tangent, as every line through that point must already meet $\K$.

\begin{Remark} \em By \cite[Corollary 13.3]{JH98}, a blocking set exists in $\PG(2,q)$ if and only if $q>2$. Hence, hereafter we always tacitly assume that $q\ge3$.
\end{Remark}

In any finite projective plane of order $q>2$, let $a, b, c$ be any three non concurrent lines. Let $\T$ be the set of all points which lie on exactly one of these lines, that is, the set of points on the sides of this triangle without the vertices. Then $\T$ establishes the existence of blocking semiovals of size $3(q-1)$ in all finite projective planes (see, for instance \cite{JMD00}) except for the \emph{Fano plane}, which does not contain blocking sets.

We close the section with some results from \cite{JMD00, JMD002, CS002}.

\begin{Theorem} \label{thm:size} Let $k$ be the size of a blocking semioval  of the projective plane $\PG(2, q)$. Then
\begin{enumerate}
\item[\em(a)] $2q+2\le k\le 3(q-1)$, if $q>5$ and the blocking semioval has the property $x_{q-1}\neq 0$, where $x_{q-1}$ denotes the number of lines of $\PG(2, q)$ which meets the blocking semioval in exactly $q-1$ points;
\item[\em(b)] there exist blocking semiovals of size $k=3q-4$, for any $q\ge5$; 
\item[\em(c)] there exist blocking semiovals of size $k=3p^e-p-2$, $p^e=q$, where $p$ is a prime number, $p\ge 3$ and $e\ge 2$.
\end{enumerate}
\end{Theorem}

\begin{Remark} \em From \cite[Theorem 3.3]{JMD00}, if $q>5$, then $2q+2$ is the lower bound for the size $k$ of any blocking semioval of the projective plane $\PG(2, q)$. 
\end{Remark}

\section{Blocking sets with the $r_{\infty}$-property in the projective planes}\label{secPG:3}

In this section we consider blocking sets that have a special property. We relax the condition in the definition of blocking semioval as follows.
\begin{Definition} \label{Def:rproperty} Let $\B$ be a blocking set in $\PG(2, q)$. We say that $\B$ has the $r_{\infty}$-property with respect to $P\in \B$ if there exists through $P$ only one tangent to $\B$, called $r_{\infty}$, and the other lines through $P$ are secants to $\B$.
\end{Definition}

If it is not necessary to specify the point $P$, we say that $\B$ has the $r_{\infty}$-property. 

\begin{Remark}\label{Rem:blocksemiovals}
	\rm Notice that the blocking semiovals $\B$ satisfy the $r_{\infty}$-property with respect to every point of $\B$. Hence, the family of the blocking semiovals is a subset of minimal blocking sets with the $r_{\infty}$-property.
\end{Remark}

The next remark will be crucial.
\begin{Remark}\label{Rem:vertexlesstriangle}
	\rm By \cite{JMD00} it is known that for any $q>2$, the vertexless triangle of size $3(q-1)$ is a blocking semioval of $\PG(2,q)$, and thus has the $r_{\infty}$-property.
\end{Remark}

In order to discuss the blocking sets in $\PG(2, 3)$ which satisfy the $r_{\infty}$-property, we quote the next definition from \cite[Page 335]{JH98}.

\begin{Definition} A projective triangle of side $n$ in $\PG(2, q)$ is a set $\B$ of $3(n-1)$ points such that:
\begin{enumerate}
\item[\em(i)] on each side of a triangle $P_0P_1P_2$ there are $n$ points of $\B$;
\item[\em(ii)] the vertices $P_0$, $P_1$, $P_2$ are in $\B$;
\item[\em(iii)] if $Q_0$ on $P_1P_2$ and $Q_1$ on $P_2P_0$ are in $\B$ then so is $Q_2=Q_0Q_1\cap P_0P_1$.
\end{enumerate}
\end{Definition}

It is known that all minimal blocking set of the projective plane $\PG(2, 3)$ have size $6$ and that they are projective triangles of side $3$ \cite[Theorem 13.21]{JH98}. With the next theorem we prove that every projective triangle of side $3$ in the projective plane $\PG(2, 3)$ is a vertexless triangle and viceversa.

\begin{Theorem} \label{thm:PG(2, 3)} In $\PG(2, 3)$ the set of projective triangles of side $3$ is the set of vertexless triangles.
\end{Theorem}
\begin{proof} Keeping in mind that in a projective plane of order $3$ the lines have four points, we consider a projective triangles of side $3$ with vertices $A, B, C$ and sides $a=(A_1BCA_2)$, $b=(AB_1CB_2)$, $c=(ABC_1C_2)$, with $A_1, B_1, C_1$ on the same line $\ell$. Then $\B =\{A, B, C, A_1, B_1, C_1\}$ is a minimal blocking set. Now the lines $a'=(AA_1KH)$, $b'=(BB_1KT)$, $c'=(CC_1HT)$ prove that $\B$ is the vertexless triangle with vertices $H, K, T$ and sides $a', b', c'$. The contrary is similar.
\end{proof}

\begin{Corollary} \label{cor:PG(2, 3)} In $\PG(2, 3)$ every minimal blocking set is a vertexless triangle and so it has the $r_{\infty}$-property.
\end{Corollary}
\begin{proof}
	The result follows from \cite[Theorem 13.21]{JH98} and Theorem \ref{thm:PG(2, 3)}.
\end{proof}
 
Now we consider the case $q\ge 4$. For our aim, we analyze some special minimal blocking sets in a projective plane of order $q\ge 4$, whose construction is due to Innamorati and Maturo \cite{IM91}. Their result is also described in \cite[Theorem 13.15]{JH98}, whose presentation we follow here. We will prove that these minimal blocking sets verify the $r_{\infty}$-property, but they are not blocking semiovals.\medskip

\begin{Setup} \label{setup} \em Let $q\ge 4$. Any construction of minimal blocking sets of size $k$ which follows is called \emph{$k$-construction} with $2q-1\le k\le 3q-5$, and we say that \emph{the blocking set is built with a $k$-construction}.

Let $\T$ be a vertexless triangle determined by the non concurrent lines $a, b, c$. Label the line intersections as follows: $a\cap b=C$, $a\cap c= B$ and $b\cap c=A$. Then the vertexless triangle $\T$ consists of the $3q-3$ points on the lines $a, b$ and $c$ less the points $A, B$ and $C$.

Choose a line $\ell$ through $A$ not a side of $\T$ and a point $D_1$ on $\ell$ but not on a side of $\T$. Let $B_1=BD_1\cap b$ and $C_1=CD_1\cap c$. Define $A'= \ell \cap a$.
\begin{enumerate}
\item[-] When $q$ is odd, let $D_2=(A'B_1\cap c)C\cap \ell$. Moreover, define $B_i=BD_i\cap b$ and $C_i=CD_i\cap c$, for $i=1, \ldots, n$. In particular, $C_2=A'B_1\cap c$. Hence, we choose $D_3, \ldots, D_n$ distinct of $\ell \setminus \T$ with $n\le q-2$.
\item[-] When $q$ is even, use the same construction for $B_i$ and $C_i$, and choose $D_2$ on $\ell$ distinct as well from the other $D_j$.
\end{enumerate}
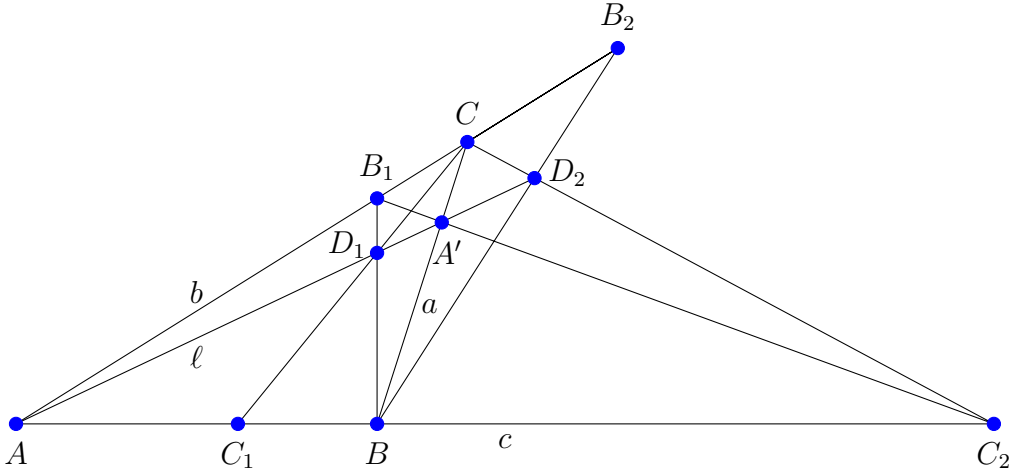
\vspace*{-0.9em}
\begin{figure}[H]
	\begin{tikzpicture}
		\draw[-] (13,0) -- (0,0) -- (8,5) -- (6,3.75) -- (13,0) -- (4.8,3) -- (4.8,0) -- (8,5) -- (6,3.75) -- (4.8,0);
		\draw[-] (0,0) -- (6.893,3.27);
		\draw[-] (2.95,0) -- (6,3.75);
		\node[fill,draw,circle,blue,label=below:$A$,minimum size=5pt, inner sep=0pt] at (0,0) {};
		\node[fill,draw,circle,blue,label=below:$C_2$,minimum size=5pt, inner sep=0pt] at (13,0) {};
		\node[fill,draw,circle,blue,label=above:$B_2$,minimum size=5pt, inner sep=0pt] at (8,5) {};
		\node[fill,draw,circle,blue,label=above:$C$,minimum size=5pt, inner sep=0pt] at (6,3.75) {};
		\node[fill,draw,circle,blue,label=above:$B_1$,minimum size=5pt, inner sep=0pt] at (4.8,3) {};
		\node[fill,draw,circle,blue,label=below:$B$,minimum size=5pt, inner sep=0pt] at (4.8,0) {};
		\filldraw[color=blue] (6.893,3.27) circle (2.5pt) node[right,xshift=1,yshift=2]{\textcolor{black}{$D_2$}};
		\node[fill,draw,circle,blue,label=below:$\,\,A'$,minimum size=5pt, inner sep=0pt] at (5.66,2.68) {};
		\filldraw[color=blue] (4.8,2.275) circle (2.5pt) node[left,yshift=3.2]{\textcolor{black}{$D_1$}};
		\node[fill,draw,circle,blue,label=below:$C_1$,minimum size=5pt, inner sep=0pt] at (2.95,0) {};
		\filldraw (6.5,0) node[below]{$c$};
		\filldraw (5.5,1.8) node[below]{$a$};
		\filldraw (2.4,1.47) node[above]{$b$};
		\filldraw (2.4,1.18) node[below]{$\ell$};
	\end{tikzpicture}
\end{figure}

Then, in \cite[Theorem 13.15]{JH98} is proved that, for all $n=2,\ldots,q-2$, the set
\[\B=(\T\cup\{D_i, i=1, \ldots, n\})\setminus\{B_i, C_i, i=1, \ldots, n\}\]
is a minimal blocking set of size 
\begin{equation} \label{eq:size}
k=3(q-1)+n-2n=3q-3-n.
\end{equation}
\end{Setup}

Note that for $n=2$ the size of $\B$ is $3q-5$ and for $n=q-2$ the size of $\B$ is $2q-1$, that is, $2q-1\le \vert \B\vert \le 3q-5$.

\begin{Lemma} \label{lem1} Let $\B$ be a minimal blocking set in $\PG(2, q)$, $q\ge 5$, built with a $k$-construction with $2q\le k\le 3q-5$, then $\B$ has the $r_{\infty}$-property with respect to a point $P\in \B$, and it is not a blocking semioval for every $k$.
\end{Lemma}
\begin{proof} Let $2q\le k \le 3q-5$. Then, $2q\le 3q-3-n\le 3q-5$, and consequently, $2\le n\le q-3$. The set $\B$ consists in $q-1$ points of the line $a$, $n$ points of the line $\ell$, and $2(q-n-1)$ points of the lines $b, c$. It follows that the line $c$ contains at least two points of $\B$. Indeed, for $n=q-3$, the line $c$ has $q-n-1=2$ points of $\B$. 

Let $P\in \B\cap c$. We prove that $PC$ is the unique tangent to $\B$ through $P$. Indeed, if $PC\cap \ell=K$, then, $K\notin \{D_i, \, i=1, \ldots, n\}$, since $D_iC\cap c=C_i\notin \B$.

It follows that $PC$ is a tangent since it intersects $\B$ in only one point $P$. Moreover, $PA$ is a secant because $c$ contains two points in $\B$, and all other lines through $P$ meet the line $a$. To prove that $\B$ is not a blocking semioval, it is sufficient to observe that through the points $D_i$ there are two tangents, that is, $BB_i$ and $CC_i$.
\end{proof}

\begin{Lemma} \label{lem2} Let $\B$ be a minimal blocking set in $\PG(2, q)$, $q \ge 4$, built with a $(2q-1)$-construction, then $\B$ does not have the $r_{\infty}$-property.
\end{Lemma}
\begin{proof} A minimal blocking set built with a $(2q-1)$-construction consists in 
$q-1$ points on the line $a$, $q-2$ points on the line $\ell$, one point on the line $b$, and one point on the line $c$. This construction determines the following minimal blocking set:
\[\B = \{a \setminus\{B, C\}\}\cup  \{\ell \setminus\{A,A',T\}\}\cup\{BT\cap AC=X, AB\cap CT=Y\},\]
where $T=\ell\setminus\{D_1,\dots,D_{q-2},A,A'\}$.

Observe that $BT, AC$ are two tangents to $\B$ through $X$, and $CT, AB$ are two tangents to $\B$ through $Y$. Moreover, for every $P\in \B\cap\{a\setminus A'\}$, $PT, PA$ are tangents and for every $P\in \B\cap\{\ell \setminus A'\}$, $PB, PC$ are tangents. Finally, since every line has at least $q+1\ge 5$ points, there exist at least two tangents through $A'$. 
\end{proof}

From Lemmas \ref{lem1} and \ref{lem2}, the next statement follows.
\begin{Theorem} \label{thm:existence}In $\PG(2, q)$, $q\ge 5$, there exists a minimal blocking set of size $k$, for every $2q\le k\le 3q-5$, with the $r_{\infty}$-property which is not a blocking semioval.
\end{Theorem}

\begin{Theorem}\label{Thm:q>=5,k>=2q} In $\PG(2, q)$, $q\ge 5$, every minimal blocking set built with a $k$-construction, with $2q\le k\le 3q-5$, has the $r_{\infty}$-property with respect to every point $P\in \B\cap (\{a\setminus \bigcup_{i,j=1}^{n}\{B_iC_j\cap a\}\}\cup b\cup c)$.
\end{Theorem}
\begin{proof} For every point $P\in \B\cap c$ the only tangent is $PC$, for every point $P\in \B\cap b$ the only tangent is $PB$. For every point $D_i$ there are two tangents, then we must exclude these points and also points $P$ of $\B\cap \{a\setminus A'\}$, which are intersection of the lines $B_iC_j$ with $a$, because for these points we have the tangents $B_iC_j$ and $PA$. Similarly, for $A'$ there are at least two tangents $B_1C_2$ and $B_2C_1$.
\end{proof}

Combining Theorem \ref{thm:existence} with Theorem \ref{thm:size}, we obtain the next result.

\begin{Theorem} \label{thm: 5} If $q\ge5$, then there exists a minimal blocking set with the $r_{\infty}$-property of size $k$ for every $k$ such that $2q\le k\le 3(q-1)$.
\end{Theorem}
\begin{proof}
	From Theorem \ref{thm:existence} the desired blocking sets exist for any $2q\le k\le 3q-5$. For $k=3q-4$ the assertion follows from Theorem \ref{thm:size}(b) together with Remark \ref{Rem:blocksemiovals}. Finally, for $k=3q-3$ we can apply Remark \ref{Rem:vertexlesstriangle}.
\end{proof}

\section{Blocking sets with the $\Pi$-property in the Affine Planes}\label{secPG:4}

We have seen in Theorem \ref{thm:existence} that for $q\ge4$, and any $2q\le k\le 3q-5$, there exists a blocking set of size $k$ in $\PG(2,q)$ having the $r_\infty$-property. The natural question arises whether there exists any blocking set in $\PG(2,q)$, $q\ge4$, with the $r_\infty$-property, having size $k\le 2q-1$. Surprisingly, we have
\begin{Theorem} \label{thm:no} Every blocking set $\B$ in $\PG(2, q)$ with size $k\le 2q-1$ does not verify the $r_{\infty}$-property with respect to any point $P\in \B$.
\end{Theorem}
\begin{Remark} \em Note that Lemma \ref{lem2} is a particular case of Theorem \ref{thm:no}.
\end{Remark}

In order to prove the theorem, firstly we investigate classes of minimal blocking (semiovals) sets in the affine plane $\AG(2, q)$. The next general lemma will be fundamental. It can be considered as a kind of permanence property.

\begin{Lemma} \label{lem:PG(2,q)} Let $\B$ be a minimal blocking set in $\PG(2, q)$, which has the $r_{\infty}$-property with respect to $P\in \B$, then $\B'=\B \setminus P$ is a minimal blocking set of $\AG(2, q) = \PG(2, q)\setminus r_{\infty}$.
\end{Lemma}
\begin{proof} Every line of $\PG(2, q)$ through $P$, which is different from $r_{\infty}$, intersects $\B'$, and every line not through $P$ intersects $\B'$. Thus $\B'$ is a blocking set of $\AG(2, q)$. Let $P'$ be a point of $\B'$, then there is a line $r$ of $\PG(2, q)$ that intersects $\B'$ exactly in $P'$ and consequently a line of $\AG(2, q)$ that intersects $\B'$ exactly in $P'$. Then $\B'$ is minimal.
\end{proof}

\begin{Example} \label{es:2} \em In $\PG(2, 3)$ with the set of points $\{ABCDEFGHIJKLM\}$, we consider the minimal blocking set $\B=\{ABCDFI\}$. Next picture describes $\PG(2,3)$. 
	
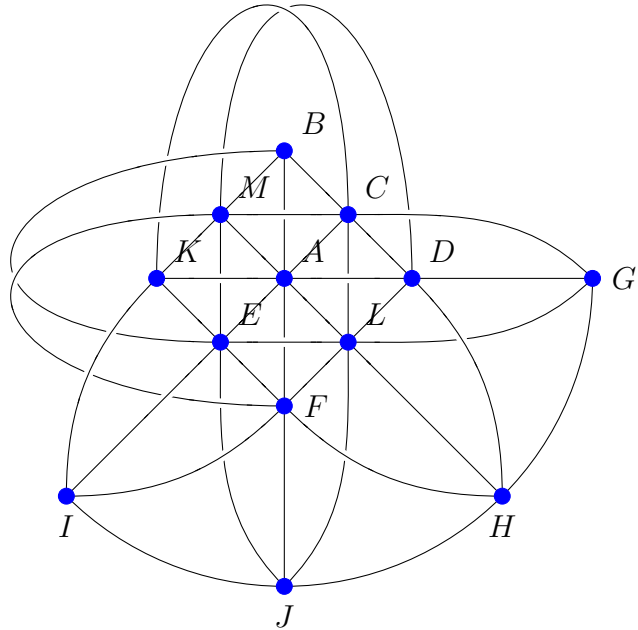
\vspace*{-4em}
	\begin{figure}[H]
		\begin{tikzpicture}[scale=1.2,rotate=-45]
			\foreach \x in {0,...,2}
			\foreach \y in {0,...,2} {
				\pgfmathtruncatemacro\nk{\y*3+\x+1}
				\coordinate (\nk) at (\x,2-\y);
			}
			\coordinate (11) at ({3+sqrt(2)},1);
			\coordinate (13) at (1,{-1-sqrt(2)});
			\coordinate (10) at ($(2,2)+(45:2)$);
			\coordinate (12) at ($(2,0)+(-45:2)$);
			\begin{knot}
				\strand (10) arc[radius={2+sqrt(2)},start angle=45,end angle=-90];
				\strand (1) -- (2) -- (3) to[out=0,in=135] (11);
				\strand (4) -- (5) -- (6) -- (11);
				\strand (7) -- (8) -- (9) to[out=0,in=-135] (11);
				\strand (1) -- (4) -- (7) to[out=-90,in=135] (13);
				\strand (2) -- (5) -- (8) -- (13);
				\strand (3) -- (6) -- (9) to[out=-90,in=45] (13);
				\strand (7) -- (5) -- (3) -- (10);
				\strand (9) .. controls +(-2.5,-2.5) and +(-2.5,-2.5) .. (4) -- (2) to[out=45,in=-180] (10);
				\strand (1) .. controls +(-2.5,-2.5) and +(-2.5,-2.5) .. (8) -- (6) to[out=45,in=-90] (10);
				\strand (1) -- (5) -- (9) -- (12);
				\strand (7) .. controls +(-2.5,2.5) and +(-2.5,2.5) .. (2) -- (6) to[out=-45,in=90] (12);
				\strand (3) .. controls +(-2.5,2.5) and +(-2.5,2.5) .. (4) -- (8) to[out=-45,in=180] (12);
			\end{knot}
			\node[fill,draw,circle,blue,label=north east:$B$,minimum size=6pt, inner sep=0pt] at (1) {};
			\node[fill,draw,circle,blue,label=north east:$C$,minimum size=6pt, inner sep=0pt] at (2) {};
			\node[fill,draw,circle,blue,label=north east:$D$,minimum size=6pt, inner sep=0pt] at (3) {};
			\node[fill,draw,circle,blue,label=north east:$M$,minimum size=6pt, inner sep=0pt] at (4) {};
			\node[fill,draw,circle,blue,label=north east:$A$,minimum size=6pt, inner sep=0pt] at (5) {};
			\node[fill,draw,circle,blue,label=north east:$L$,minimum size=6pt, inner sep=0pt] at (6) {};
			\node[fill,draw,circle,blue,label=north east:$K$,minimum size=6pt, inner sep=0pt] at (7) {};
			\node[fill,draw,circle,blue,label=north east:$E$,minimum size=6pt, inner sep=0pt] at (8) {};
			\node[fill,draw,circle,blue,label=right:$F$,minimum size=6pt, inner sep=0pt] at (9) {};
			\node[fill,draw,circle,blue,label=right:$G$,minimum size=6pt, inner sep=0pt] at (10) {};
			\node[fill,draw,circle,blue,label=below:$H$,minimum size=6pt, inner sep=0pt] at (11) {};
			\node[fill,draw,circle,blue,label=below:$J$,minimum size=6pt, inner sep=0pt] at (12) {};
			\node[fill,draw,circle,blue,label=below:$I$,minimum size=6pt, inner sep=0pt] at (13) {};
		\end{tikzpicture}
	\end{figure}
	\vspace*{-1em}

From Corollary \ref{cor:PG(2, 3)}, since $\mathcal{B}$ is a vertexless triangle, then it is a blocking semioval and consequently verifies the $r_{\infty}$-property. The sides of the triangle are:
\[a=(BKIM),\,\,\, b=(ADGK),\,\,\, c=(CFGM).\]
 
We consider the line $r_{\infty}=(GJHI)$ and the affine plane $\AG(2, 3) = \PG(2, 3)\setminus r_{\infty}$ which is depicted below.
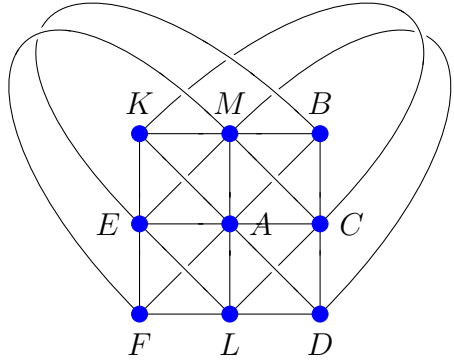
\vspace*{-3em}
\begin{figure}[H]
	\begin{tikzpicture}[scale=1.2,rotate=-90]
		\foreach \x in {0,...,2}
		\foreach \y in {0,...,2} {
			\pgfmathtruncatemacro\nk{\y*3+\x+1}
			\coordinate (\nk) at (\x,2-\y);
		}
		\begin{knot}
			\strand (1) -- (2) -- (3);
			\strand (4) -- (5) -- (6);
			\strand (7) -- (8) -- (9);
			\strand (1) -- (4) -- (7);
			\strand (2) -- (5) -- (8);
			\strand (3) -- (6) -- (9);
			\strand (7) -- (5) -- (3);
			\strand (9) .. controls +(-2.5,-2.5) and +(-2.5,-2.5) .. (4) -- (2);
			\strand (1) .. controls +(-2.5,-2.5) and +(-2.5,-2.5) .. (8) -- (6);
			\strand (1) -- (5) -- (9);
			\strand (7) .. controls +(-2.5,2.5) and +(-2.5,2.5) .. (2) -- (6);
			\strand (3) .. controls +(-2.5,2.5) and +(-2.5,2.5) .. (4) -- (8);
		\end{knot}
		\node[fill,draw,circle,blue,label=above:$B$,minimum size=6pt, inner sep=0pt] at (1) {};
		\node[fill,draw,circle,blue,label=right:$C$,minimum size=6pt, inner sep=0pt] at (2) {};
		\node[fill,draw,circle,blue,label=below:$D$,minimum size=6pt, inner sep=0pt] at (3) {};
		\node[fill,draw,circle,blue,label=above:$M$,minimum size=6pt, inner sep=0pt] at (4) {};
		\node[fill,draw,circle,blue,label=right:$A$,minimum size=6pt, inner sep=0pt] at (5) {};
		\node[fill,draw,circle,blue,label=below:$L$,minimum size=6pt, inner sep=0pt] at (6) {};
		\node[fill,draw,circle,blue,label=above:$K$,minimum size=6pt, inner sep=0pt] at (7) {};
		\node[fill,draw,circle,blue,label=left:$E$,minimum size=6pt, inner sep=0pt] at (8) {};
		\node[fill,draw,circle,blue,label=below:$F$,minimum size=6pt, inner sep=0pt] at (9) {};
	\end{tikzpicture}
\end{figure}
Then $\B'=\B\setminus I=\{ABCDF\}$ is a minimal blocking set of size $5$ of the lines of the affine plane, which are:
\[(ADK), (CFM), (CKL), (BCD), (ACE), (ALM), (ABF), (BKM),\]
\[ (EFK), (DFL), (DEM), (BEL).\]
\end{Example}

Example \ref{es:2} highlights that in the affine plane $\AG(2, 3)$ we have minimal blocking sets of size $5$, since we have only the vertexless triangles whose size is $3(q-1)=6$ as  minimal blocking sets in the projective plane $\PG(2, 3)$.\medskip

Now we examine the existence of minimal blocking sets in $\AG(2, q)$, $q\ge 5$. Combining Theorems \ref{thm:existence} and \ref{thm: 5} with Lemma \ref{lem:PG(2,q)} we obtain the following two results.

\begin{Theorem} In $\AG(2, q)$, $q\ge 5$, there exists a minimal blocking set of size $k$ for every $k$ such that $2q-1\le k \le 3q-6$.
\end{Theorem}
\begin{Theorem} In $\AG(2, q)$, if $q\ge 5$ is odd, there exists a minimal blocking set of size $k$ for every $k$ such that $2q-1\le k \le 3q-4$.
\end{Theorem}

In what follows, sometimes to avoid any ambiguity we will speak about a \emph{projective blocking set} when we refer to a blocking set of the projective plane $\PG(2,q)$, and of an \emph{affine blocking set} when we refer to a blocking set of the affine plane $\AG(2,q)$.\medskip

In order to introduce some special affine blocking sets and to highlight their connection with some projective blocking sets, we give the next definition.

\begin{Definition} A parallel class in an affine plane $\AG(2, q)$ is the set of all lines parallel to a line $\ell$. We denote it by $\Pi_\ell$ and call it the direction $\Pi_\ell$.
\end{Definition}

For instance, in Example \ref{es:2}, the lines $(ALM)$, $(BCD)$, $(EFK)$ of $\AG(2,3)$ are parallel.

In what follows we will regard any parallel class, generated by a line $\ell$ of $\AG(2, q)$, as a new point. 

Furthermore, we collect all new points into a new line $r_{\infty}$. Then we can define the projective plane $\PG(2, q)$ as follows: $\PG(2, q)= \AG(2, q)\cup r_{\infty} $. See for instance \cite{AP2018} and the reference therein.

We will indicate a parallel class also by $\Pi_P$ if it is the parallel class of all lines passing through $P\in r_{\infty}$ in $\PG(2, q)$.

\begin{Definition} \label{def:prodproperty} A blocking set $\B$ in $\AG(2, q)$ has the $\Pi$-property with respect to a direction $\Pi_\ell$ if there exists a parallel class $\Pi_\ell$ verifying the following conditions:
\begin{enumerate}
\item[\em(j)] through every point $Q\in \B$ there exists a line $m\notin \Pi_\ell$ tangent to $\B$,
\item[\em(jj)] not one line of $\Pi_\ell$ is contained in $\B$.
\end{enumerate}
\end{Definition}

Now with (j') we indicate the following condition that implies the condition (j):

\begin{enumerate}
\item[(j')] \emph{through every point $Q\in \B$ there exists a unique tangent $m\notin \Pi_\ell$ to $\B$}.\medskip
\end{enumerate}

Blocking sets that verify the conditions (j) and (jj) in Definition \ref{def:prodproperty} are necessarily minimal and blocking set that verify the conditions (j') and (jj) are also maximal.

\begin{Example}\label{es:3}\em Let $a, b$ be two parallel lines of $\AG(2, q)$ and let $c$ be another line not parallel to $a, b$. Set $A=b\cap c$ and $B=a\cap c$. Then $\B = (a\cup b\cup c) \setminus \{A, B\}$ is a minimal blocking set of $\AG(2, q)$ of size $3q-4$ that has the $\Pi$-property with respect to the direction $\Pi_c$. Note that in this example every line in $\Pi_c$ is a secant to $\B$, and the conditions (j') and (jj) are verified.
\end{Example}

\begin{Example}\label{es:4}\em Let $a, b$ be two parallel lines of $\AG(2, q)$, $q\ge 5$, and let $c$ be another line not parallel to $a, b$. Set $A=b\cap c$ and $B=a\cap c$. Choose a line $\ell$ through $A$ not the line $c$ or $b$ and a point $D_1$ on $\ell$ but not on $b$ or $a$. Let
\[A'=\ell\cap a, \, C_1=r\cap c, \, B_1=BD_1\cap b, \, C_2=A'B_1\cap c, \, D_2=s\cap \ell, \, B_2=BD_2\cap b,\]
with $r$ the parallel line through $D_1$ to $a$ and $s$ the parallel line through $C_2$ to $a$. Then $\B = ((a\cup b\cup c) \cup \{D_1, D_2\})\setminus \{A, B, B_1, B_2, C_1, C_2\}$ is a minimal blocking set of $\AG(2, q)$ of size $3q-6$, that has the $\Pi$-property with respect to the direction $\Pi_c$. The condition (j') is not verified, because for $D_1, D_2$ there are two tangents.
\end{Example}

Now we fix a line $\ell = r_{\infty}$ in the projective plane $\PG(2, q)$ and a point $P\in \ell$. We indicate with $\X_{P, \ell}$ the set of all minimal blocking sets in $\PG(2, q)$ with the $r_{\infty}$-property with respect to $P$, such that the unique tangent to any of these blocking sets is the line $\ell$, and with $\Y_{\Pi_P}$ the set of all minimal blocking sets in $\AG(2, q)= \PG(2, q)\setminus \ell$ with the $\Pi$-property with respect to a direction $\Pi_P$.

Next, we show that every blocking set in $\PG(2, q)$ with the $r_{\infty}$-property with respect to $P\in r_{\infty}$ determines a minimal blocking set in $\AG(2, q)= \PG(2, q)\setminus r_{\infty}$ with the $\Pi$-property with respect to a direction $\Pi_P$ and viceversa.

\begin{Proposition} \label{propo:alpha}The map $\alpha: \X_{P, \ell} \longrightarrow \Y_{\Pi_P}$ defined as follows
\[\B\in \X_{P, \ell} \longmapsto \B\setminus P \in \Y_{\Pi_P}\]
is a bijection.
\end{Proposition}
\begin{proof} From Lemma \ref{lem:PG(2,q)}, it follows that $\B\setminus P$ is a minimal blocking sets of $\AG(2, q)$, and from the definition of $\alpha$, it follows that the mapping is injective.

Now we prove that $\B\setminus P$ has the $\Pi$-property with respect to a direction $\Pi_P$. 

Suppose that $\B\setminus P$ contains a line in $\Pi_P$, then $\B$ contains a line in $\PG(2, q)$ and $\B$ is not a blocking set. This verifies the condition (jj) of Definition \ref{def:prodproperty}. Let $\ell' \in \Pi_P$ be a line tangent to $\B\setminus P$ in $\AG(2, q)$, then $\ell'$ is not tangent to $\B$ in $\PG(2, q)$. Consequently there is a tangent $m\notin \Pi_P$ to $\B$ passing through $Q\in \ell' \cap (\B\setminus P)$. This line $m$ is also a tangent to $\B\setminus P$, since it does not contain the point $P$, and this implies the condition (j) of Definition \ref{def:prodproperty}. Finally, the claim is proved.

Now we start from a minimal blocking set $\B'\in \Y_{\Pi_P}$ of $\AG(2, q)$. It is trivial that $\B= \B'\cup P$ is a minimal blocking set in $\PG(2, q)$. It contains a line, this line must pass through $P$, and consequently there is a line in $\Pi_P$ which is contained in $\B'$, that contradicts (jj) of Definition \ref{def:prodproperty}. This implies that $\B$ is a blocking set of $\PG(2, q)$. For every point of $\B'$, there exists a tangent to $\B'$ in the affine plane that is not in $\Pi_P$ and thus the same tangent to $\B$ in the projective plane. Through $P$ there exists the unique tangent $\ell$ to $\B$ and this means that $\B$ has the 
$r_{\infty}$-property with respect to $P$, and the unique tangent to $\B$ is the line $\ell$. To finish the proof $\alpha(\B)=\B'$.
\end{proof}

According to the definition given in the projective case (Setup \ref{setup}), we say that a minimal blocking set $\B'$ of $\AG(2, q)$ is built with a $k$-construction, $2q\le k\le 3q-5$, if considered a minimal blocking set $\B$ of $\PG(2, q)$ built with a $k$-construction, chosen a point $P$ for which $\B$ has the $r_{\infty}$-property and such that the unique tangent through $P$ to $\B$ is $\ell$, we have $\alpha(\B)=\B'$. 

Observe that Example \ref{es:3} and Example \ref{es:4} are built starting from a vertexless triangle in the projective plane of the order $q$, and from a $(3q-5)$-construction, respectively.

\begin{Corollary}
Let $\B$ be a minimal blocking set in $\AG(2, q)$, $q\ge 5$, built with a $k$-construction with $2q\le k\le 3q-5$. Then $\B$ has the $\Pi$-property with respect to a direction $\Pi_\ell$.
\end{Corollary}
\begin{proof} The proof follows from Lemma \ref{lem1} and Proposition \ref{propo:alpha}.
\end{proof}

The following key result was shown by Jamison \cite{Jamison} and, with a simpler proof, by Brouwer and Schrijver \cite{BS78}. See also \cite[Corollary 13.46]{JH98}.
\begin{Lemma} \label{lem:[7]}
Every blocking set of the lines in $\AG(2, q)$ has a size greater than or equal to $2q-1$.
\end{Lemma}

We are now in the position to prove our main result in the section.
\begin{proof}[Proof of Theorem \ref{thm:no}] Let $\B$ be a blocking set with the $r_{\infty}$-property with respect to a point $P$ of size $k\le 2q-1$. Since there exists a tangent through $P$ to $\B$, any blocking set $\B'$ with $\B'\subset\B$ must contain $P$. Therefore there exists a minimal blocking set $\B'\subset\B$ with the $r_{\infty}$-property with respect to the same point $P$. From Proposition \ref{propo:alpha}, there exists a blocking set $\B' \setminus P$ in $\AG(2, q)$ with size $k'< 2q-1$, which contradicts Lemma \ref{lem:[7]}.
\end{proof}

\begin{Definition} A blocking set $\B$ in $\AG(2, q)$ has the $\Pi$-strong property with respect to a direction $\Pi_\ell$ if the conditions \textup{(j')} and \textup{(jj)} are verified.
\end{Definition}

Let us denote with $\Z_{P, \ell}$ the set of all blocking semiovals in $\PG(2, q)$ which contain the point $P$ and such that $\ell$ is the unique tangent to $\B$ through $P$, and with $\W_{P, \ell}$ the set of all blocking semiovals in $\AG(2, q)= \PG(2, q)\setminus \ell$ with the $\Pi$-strong property with respect to a direction $\Pi_P$.

The next result shows that every $\B\in \Z_{P, \ell}$ determines a blocking semioval in $\AG(2, q)= \PG(2, q)\setminus \ell$ with the $\Pi$-property with respect to a direction $\Pi_P$.

\begin{Proposition} Let $\alpha: \X_{P, \ell} \longrightarrow \Y_{\Pi_P}$ be the bijection defined as follows:
\[\B\in \X_{P, \ell} \longmapsto \B\setminus P \in \Y_{\Pi_P}.\]
Then $\alpha(\Z_{P, \ell}) = \W_{P, \ell}$.
\end{Proposition}
\begin{proof} Let $\B$ be a blocking semioval. From Proposition \ref{propo:alpha}, it is sufficient to prove condition (j'). For every $Q\in \B$ ($Q\neq P$) there exists a unique tangent $m$ to $\B$. This tangent does not contain $P$, then $m\notin \Pi_\ell$ in the affine plane. Starting from a minimal blocking set $\B'$ in $\AG(2, q)$ that verifies the conditions (j') and (jj), then for every point $Q\in \B'$ the unique tangent to $\B'$ is also the unique tangent to $\B=\B'\cup P$ in $\PG(2, q)$. At the end for $P$ the unique tangent is $\ell$. This proves that $\B=\B'\cup P$ in $\PG(2, q)$ is a blocking semioval.
\end{proof}

\section{Conclusions and Open questions}\label{secPG:5}

In the previous sections we have discussed the existence of certain blocking sets with the $r_\infty$-property in $\PG(2, q)$ having a given size.\medskip

Let us consider the case $q=4$.

In $\PG(2, 4)$ we have only minimal blocking sets of size $k\in\{7, 8, 9\}$. For $k=9$, there are the vertexless triangles which have the $r_{\infty}$-property. If $k=7$, there are no blocking sets with the $r_{\infty}$-property (Theorem \ref{thm:no}). 

The next example, provided by one of the referees, shows that there exists a blocking set with the $r_\infty$-property of size $3q - 4$ in $\PG(2,q)$ for $q=4$.

Recall that a proper subplane $\B$ of a projective plane $\PG(2, q)$ is called a \textit{Baer subplane} if each line of $\PG(2, q)$ contains a point in $\B$ and, dually, each point of $\PG(2, q)$ is incident with a line in $\B$.

\begin{Example}\label{es:5} \em Let $R$ and $T$ be two points of a Baer subplane $\B$ in $\PG(2, 4)$. Let $t$ be one of the two tangent lines to $\B$ at $T$ and $r_1, r_2$ be the two tangent lines to $\B$ at $R$. Finally, for $i = 1,2$, denote the point $t\cap  r_i$ by $T_i$. Then it is easy to check that $(\B\setminus \{R\})\cup\{T_1, T_2\}$ is a minimal blocking set of cardinality $8 = 3q -4$ and it has the $r_\infty$-property at $T$. Indeed, the unique tangent is the other tangent to $\B$ at $T$.
\end{Example}

Let $q\ge5$. By Theorem \ref{thm:no}, for any $1\le k\le 2q-1$, no blocking set with the $r_\infty$-property of size $k$ in $\PG(2,q)$ exists. Whereas, by Theorem \ref{thm:existence}, there exists a blocking set $\B$ with the $r_\infty$-property of size $k$ in $\PG(2,q)$, for all $2q\le k\le 3q-5$ and the size $k=3(q-1)$ is achieved if $\B$ is a vertexless triangle. Furthermore, by Theorem \ref{thm:size}, if $q\ge5$, there exists such a set $\B$, in fact a blocking semioval, of size $k=3q-4$. 

These considerations lead to the following open question.\smallskip

\begin{Open question} Are there blocking sets with the $r_{\infty}$-property of size $k>3(q-1)$ in $\PG(2,q)$ with $q\ge5$?
\end{Open question}\smallskip

By Theorem \ref{thm:size}(a), we know that $|\B|\le3(q-1)$ for $q>5$ and any blocking semioval $\B$ in $\PG(2,q)$ with $x_{q-1}\neq 0$. For instance any vertexless triangle which is a blocking semioval of size $3(q-1)$ satisfy the property $x_{q-1}\neq 0$. For $2q+2\le k\le 3q-5$ we do not know whether there exist blocking semiovals of size $k$. By Lemma \ref{lem1}, for $2q\le k\le 3q-5$, we only know that any blocking set built with a $k$-construction verifies the $r_\infty$-property but is not a semioval.

\begin{Open question} Determine those integers $2q+2\le k\le 3q-5$, $q>5$, for which there exists a blocking semioval of size $k$ in $\PG(2,q)$.
\end{Open question}

The case $k=3q-4$, with $q\ge5$, was addressed by Dover and Suetake \cite{JMD00, CS002}.\bigskip

{\bf Note.} \emph{This version of the article differs from the earlier one thanks to the helpful comments provided by Prof. Jeremy Dover. We are grateful for his attention to the manuscript. The revisions affect the final part of the introduction, the statement of Theorem \ref{thm:size}(a), and the concluding part of Section \ref{secPG:5}.}\bigskip

\textbf{Acknowledgments.} The authors thank the anonymous referees for their careful reading and helpful suggestions that allowed to improve the quality of the paper. The authors acknowledge support of the GNSAGA of INdAM (Italy). \\
A. Ficarra was partly supported by the Grant JDC2023-051705-I funded by
MICIU/AEI/10.13039/501100011033 and by the FSE+.

\end{document}